\documentclass[12pt]{article}
\usepackage{pifont}

\usepackage{amsfonts}
\usepackage{latexsym}
\usepackage{amsmath}
\usepackage{amssymb}
\usepackage{color}

\newtheorem{theorem}{Theorem}[section]

\newtheorem{proposition}{Proposition}[section]

\newenvironment{proof}[1][Proof]{\noindent \textbf{#1.} }{\ \ \  $\Box$}

\newtheorem{definition}{Definition}[section]

\title{$L^{p}$ solutions of backward stochastic Volterra integral equations \thanks{This work is supported by National
Natural Science Foundation of China Grant 10771122, Natural Science
Foundation of Shandong Province of China Grant Y2006A08 and National
Basic Research Program of China (973 Program, No. 2007CB814900).}}

\date{December 14 2009}

\author{Tianxiao Wang \thanks{Corresponding author, E-mail:xiaotian2008001@gmail.com}\\ \small{School of
Mathematics, Shandong University, Jinan 250100, China}}

\begin{document}

\maketitle

\begin{abstract}
This paper is devoted to the unique solvability of backward
stochastic Volterra integral equations (BSVIEs for short), in terms
of both M-solution introduced in \cite{Y3} and the adapted solutions
in \cite{L}, \cite{WS}. We prove the existence and uniqueness of
M-solutions of BSVIEs in $L^p$ $(1<p<2)$, which extends the results
in \cite{Y3}. The unique solvability of adapted solutions of BSVIEs
in $L^p$ $(p>1)$ is also considered, which also generalize the
results in \cite{L} and \cite{WS}.
\par  $\textit{Keywords:}$ Backward stochastic Volterra integral
equations, M-solutions, $L^{p}$ solutions, adapted solutions
\end{abstract}



\section{Introduction}\label{sec:intro}
In this paper, we are concerned with backward stochastic Volterra
integral equation (BSVIE for short) of the formt, $\in [0,T],$
\begin{eqnarray}
Y(t)=\psi
(t)+\int_t^Tg(t,s,Y(s),Z(t,s),Z(s,t))ds-\int_t^TZ(t,s)dW(s),
\end{eqnarray}
where $W$ is a standard Brownian motion with values in $R^d$ defined
on some complete probability space $(\Omega ,\mathcal{F},P)$, $\psi
(\cdot )$ is the terminal condition and $g$ is the coefficient (also
called the generator). The unknowns are the processes $(Y(\cdot ),Z(\cdot ,\cdot ))\in {\mathcal{H}}%
^p[0,T]$ (defined blew), for which $(Y(\cdot ),Z(t,\cdot ))$ is $\Bbb{F}$%
-adapted for all $t\in [0,T].$ Here $\{{\mathcal{F}}_{t}\}_{t\geq
0}$ is the augmented natural filtration of $W$ which satisfies the
usual conditions.

Lin \cite{L} firstly considered the solvability of the adapted
solution of the form
\begin{eqnarray}
Y(t)=\xi +\int_t^Tg(t,s,Y(s),Z(t,s))ds-\int_t^TZ(t,s)dW(s),\quad
t\in [0,T].
\end{eqnarray}
As to the general form (1), Yong (\cite{Y1} and \cite{Y3}) firstly
studied them and gave the application in optimal control. One can
also see \cite{WS} for more detailed accounts on BSVIEs (1). Both of
them are natural generalization of backward stochastic differential
equation (BSDE for short) of the form
\[
Y(t)=\xi +\int_t^Tg(s,Y(s),Z(s))ds-\int_t^TZ(s)dW(s),\quad t\in
[0,T],
\]
which was firstly introduced by Pardoux and Peng \cite{PP1}. In
order to obtain a stochastic maximum principle for optimal control
of stochastic Volterra integral equation, Yong \cite{Y3} introduced
the notion
of M-solution and proved the existence and uniqueness of M-solution in $%
{\mathcal{H}}^2[0,T].$ On the other hand, Yong \cite{Y2} adopted the
so-called BSVIEs to construct a class of dynamic convex and coherent
risk measures, which is different from the result in the case of
BSDE in Rosazza \cite{R2}. We would like to mention that Wang and
Zhang \cite{WZ} gave existence and uniqueness of the adapted
solution of BSVIE (2) with jumps under non-Lipschitz condition and
Ren considered the similar result in hilbert space in \cite{R1}. See
Aman \cite{A} for result of BSVIE (2) under local Lipschitz
condition. All of them considered the unique solvability of BSVIEs
when $p=2$.

In this paper, for mathematical interest, we try to get the
existence and uniqueness of $L^p$ $(1<p<2)$ solutions for BSVIEs. In
BSDEs case, we have to mention that El Karoui et al. \cite{EPQ}
obtained the existence and uniqueness of the solution for BSDEs when
the generator is uniformly Lipschitz, the data $\xi $ and
$\{g(s,0,0)\}$ are in $L^p$ $(1<p<2).$ Briand and Carmona \cite{BC}
considered the $L^p$ solution for BSDEs with polynomial growth
generators and in Briand et al. \cite{BDHPS} generalized the result.
Recently, Wang et al. \cite{WRC} also studied the $L^p$ solutions
when the uniform Lipschitz condition was replaced by nonnegative
adapted process. However, neither of the above can be applied in
BSVIEs case, and we have to deal with this problem with new method.
Inspired by the four steps to solve the existence and uniqueness of
M-solution in ${\mathcal{H}}^2[0,T]$ in \cite{Y3}, we will use the
similar method to deal with the situation for M-solution in
${\mathcal{H}}^p[0,T]$. Similarly we can
also get the result of adapted solutions for (1) when $g$ is independent of $%
Z(s,t).$

 The paper is organized as follows. In Section 2, we will present
 some notations, definition and some propositions. In Section 3, we
 will study the existence and uniqueness of M-solution of (1) and
 adapted solution of (1) (the generator is independent of $Z(s,t)$) respectively.

\section{Preliminary}
 In this section, we
will present some necessary notations, definitions and some
propositions needed in the sequel. In the following we denote
$\Delta^c=\Delta^c[0,T]$ and $\Delta ^c[R,S]=\{(t,s)\in [R,S]^2,$
$t\leq s\}$ where $R,$ $S\in [0,T].$ Let
$L^p_{\mathcal{F}_{T}}[R,S]$
be the set of ${\mathcal{B}}[R,S]\times \mathcal{F}_T$-measurable processes $%
\psi :\Omega \times [R,S]\rightarrow R^m$ such that $E\int_R^S|\psi
(t)|^pdt<\infty .$ $L_{\Bbb{F}}^p[R,S]$ is set of adapted processes $%
X:\Omega \times [R,S]\rightarrow R^m$ such that
$E\int_R^S|X(t)|^pdt<\infty . $

We denote $$\Bbb{H}^p[R,S]=L^p_{\Bbb{F}}(\Omega; C[R,S])\times
L^p_{\Bbb{F}}[R,S],$$ which is a Banach space under the norm
\begin{eqnarray*}
\Vert (y(\cdot ),z(\cdot ))\Vert _{\Bbb{H}^p[R,S]}=\left[
E\sup_{t\in [R,S]}|y(t)|^p+E\left(\int_R^S|z(t)|^2dt\right)^{\frac p
2}\right] ^{\frac 1 p},
\end{eqnarray*}
where $L^p_{\Bbb{F}}(\Omega; C[R,S])$ is set of all continuous
adapted processes $X :[R,S]\times\Omega\rightarrow R^m$ such that
$E\left[\sup\limits_{t\in [R,S]}|X(t)|^p\right] <\infty.$

$L^p(R,S;L_{\Bbb{F}}^2[R,S])$ the set of processes $Z:\Omega \times
[R,S]\times [R,S]\rightarrow R^{m\times d}$ such that for almost
every $t\in
[R,S],$ $Z(t,\cdot )$ is $\Bbb{F}$-progressively measurable and $$%
E\int_R^S\left( \int_R^S|Z(t,s)|^2ds\right) ^{\frac p2}dt<\infty .$$

$%
L^p(R,S;L_{\Bbb{F}}^2[t,T])$ the set of processes $Z:\Omega \times
\Delta^c \rightarrow R^{m\times d}$ such that for almost every $t\in [R,S],$ $%
Z(t,\cdot )$ is $\Bbb{F}$-progressively measurable and
$$E\int_R^S\left( \int_t^T|Z(t,s)|^2ds\right) ^{\frac p2}dt<\infty
.$$ We denote
\begin{eqnarray*}
{\mathcal{H}}^p[R,S] &=&L_{\Bbb{F}}^p[R,S]\times
L^p(R,S;L_{\Bbb{F}}^2[R,S]),
\\
{\mathcal{H}}_0^p[R,S] &=&L^p[R,S]\times
L^p(R,S;L_{\Bbb{F}}^2[t,S]).
\end{eqnarray*}
We also need the following two definitions.
\begin{definition}
Let $S\in [0,T]$. A pair of $(Y(\cdot ),Z(\cdot ,\cdot ))\in {\mathcal{H}}%
^p[S,T]$ is called an adapted $M$-solution of BSVIE (1) on $[S,T]$
if (1) holds in the usual It\^o's sense for almost all $t\in [S,T]$
and, in addition, the following holds:
\[
Y(t)=E^{\mathcal{F}_S}Y(t)+\int_S^tZ(t,s)dW(s),\quad t\in [S,T].
\]
\end{definition}
\begin{definition}
A pair of $(Y(\cdot ),Z(\cdot ,\cdot ))\in {\mathcal{H}}_0^p[0,T]$
is called an adapted solution of the following simple BSVIE (3) if
(3) holds in the usual It\^o's sense,
\begin{equation}
Y(t)=\psi (t)+\int_t^Tg(t,s,Y(s),Z(t,s))ds-\int_t^TZ(t,s)dW(s),\quad
t\in [0,T].
\end{equation}
\end{definition}
Next we will give some propositions, which can be seen in \cite{Y3}.
For any $R,S\in [0,T],$ let us consider the following stochastic
integral equation, $r\in [S,T],$ $t\in [R,T],$
\begin{equation}
\lambda (t,r)=\psi (t)+\int_r^Th(t,s,\mu (t,s))ds-\int_r^T\mu
(t,s)dW(s),
\end{equation}
where $h:[R,T]\times [S,T]\times R^{m\times d}\times \Omega
\rightarrow R^m$ is given. The unknown processes are $(\lambda
(\cdot ,\cdot ),\mu (\cdot ,\cdot ))$,
for which $%
(\lambda (t,\cdot ),\mu (t,\cdot ))$ are $\Bbb{F}$-adapted for all
$t\in [R,T]$. We can regard (4) as a family of BSDEs on $[S,T]$,
parameterized by $t\in [R,T]$. Next we introduce the following
assumption of $h$ in (4).

(H1) Let $R,S\in [0,T]$ and $h:[R,T]\times [S,T]\times R^{m\times
d}\times \Omega \rightarrow R^m$ be ${\mathcal{B}}([R,T]\times
[S,T]\times R^{m\times
d})\otimes \mathcal{F}_T$-measurable such that $s\mapsto h(t,s,z)$ is $%
\Bbb{F}$-progressively measurable for all $(t,z)\in [R,T]\times
R^{m\times d}$ and
\begin{equation}
E\int_R^T\left( \int_S^T|h(t,s,0)|ds\right) ^pdt<\infty .
\end{equation}
Moreover, the following holds:
\begin{equation}
|h(t,s,z_1)-h(t,s,z_2)|\leq L(t,s)|z_1-z_2|,(t,s)\in [R,T]\times
[S,T],z_1,z_2\in R^{m\times d},
\end{equation}
where $L:[R,T]\times [S,T]\rightarrow [0,\infty )$ is a
deterministic function such that for some $\varepsilon >0$,
$$\sup\limits_{t\in [R,T]}\int_S^TL(t,s)^{2+\epsilon }ds<\infty .$$
\begin{proposition}
Let (H1) hold, then for any $\psi (\cdot )\in
L_{\mathcal{F}_T}^p[R,T]$, (4) admits a unique adapted solution
$(\lambda (t,\cdot ),\mu (t,\cdot ))\in \Bbb{H}^p[S,T]$ for almost
all $t\in [R,T].$
\end{proposition}

Now we look at one special case of (4). Let $R=S$ and define
\begin{equation}
\left\{
\begin{array}{lc}
Y(t)=\lambda (t,t), & t\in [S,T], \\
Z(t,s)=\mu (t,s), & (t,s)\in \Delta ^c[S,T].%
\end{array}
\right.
\end{equation}
Then the above (4) reads:
\begin{equation}
Y(t)=\psi (t)+\int_t^Th(t,s,Z(t,s))ds-\int_t^TZ(t,s)dW(s),\quad t\in
[S,T].
\end{equation}

Here we define $Z(t,s)$ for $(t,s)\in \Delta[S,T]$ by
$Y(t)=EY(t)+\int_0^tZ(t,s)dW(s)$. So we have,
\begin{proposition}
Let (H1) hold, then for any $S\in [0,T],$ $\psi (\cdot )\in L_{\mathcal{F}%
_T}^p[S,T]$, (8) admits a unique adapted M-solution $(Y(\cdot
),Z(\cdot ,\cdot ))\in {\mathcal{H}}^p[S,T]$ (adapted solution
$(Y(\cdot ),Z(\cdot ,\cdot ))\in {\mathcal{H}}^p_0[S,T]$,
respectively).
If $\overline{h}$ also satisfies (H1), $\overline{\psi }(\cdot )\in L_{%
{\mathcal{F}}_{T}}^{p}[S,T],$ and $(\overline{Y}(\cdot
),\overline{Z}(\cdot ,\cdot ))\in {\mathcal{H}}^{p}[S,T]$ is the
unique adapted M-solution of BSVIE
(8) with $(h,\psi )$ replaced by $(\overline{h},\overline{\psi }),$ then $%
\forall t\in \lbrack S,T],$
\begin{eqnarray}
&&\ E\left\{ |Y(t)-\overline{Y}(t)|^{p}+\left(\int_{t}^{T}|Z(t,s)-\overline{Z}%
(t,s)|^{2}ds\right)^{\frac p 2}\right\}  \nonumber \\
\  &\leq &CE\left[ |\Psi (t)-\overline{\Psi }(t)|^{p}+\left(
\int_{t}^{T}|h(t,s,Z(t,s))-\overline{h}(t,s,Z(t,s))|ds\right)
^{p}\right] .
\end{eqnarray}%
Hereafter $C$ is a generic positive constant which may be different
from line to line.
\end{proposition}

Let's give another special case. Let $r=S\in [R,T]$ be fixed. Define
\begin{eqnarray*}
\psi ^S(t)=\lambda (t,S),\quad Z(t,s)=\mu (t,s), \quad t\in [R,S],
s\in [S,T].
\end{eqnarray*}
Then (4) becomes: $t\in [R,S],$
\begin{equation}
\psi ^S(t)=\psi (t)+\int_S^Th(t,s,Z(t,s))ds-\int_S^TZ(t,s)dW(s),
\end{equation}
and we have the following result.
\begin{proposition}
Let (H1) hold, then for any $\psi (\cdot )\in
L_{\mathcal{F}_T}^p[R,S]$, $(10)$ admits a unique adapted solution
$(\psi ^S(\cdot ),Z(\cdot ,\cdot ))\in
L_{\mathcal{F}_S}^p[R,S]\times L^p(R,S;L_{\Bbb{F}}^2[S,T])$.
\end{proposition}

\section{L$^p$ solution for BSVIE}
In this section, we will make use of the above propositions to give
the unique existence of M-solution of (1) and adapted solutions of
(3). First we assume,

 (H2) Let $g:\Delta ^c\times R^m\times
R^{m\times d}\times R^{m\times d}\times \Omega \rightarrow R^m$ be
${\mathcal{B}}(\Delta ^c\times R^m\times R^{m\times d}\times
R^{m\times d})\otimes
\mathcal{F}_T$-measurable such that $s\rightarrow g(t,s,y,z,\zeta )$ is $%
\Bbb{F}$-progressively measurable for all $(t,y,z,\zeta )\in
[0,T]\times R^m\times R^{m\times d}\times R^{m\times d}$ and
\[
E\int_0^T\left( \int_t^T|g_0(t,s)|ds\right) ^pdt<\infty,
\]
where $g_0(t,s)=g(t,s,0,0,0).$ Moreover, it holds
\[
|g(t,s,y,z,\zeta )-g(t,s,\overline{y},\overline{z},\overline{\zeta
})|\leq
L_1(t,s)|y-\overline{y}|+L_2(t,s)|z-\overline{z}|+L_3(t,s)|\zeta -\overline{%
\zeta }|,
\]
$\forall (t,s)\in \Delta ^c$, $y,\overline{y}\in R^m$,
$z,\overline{z},\zeta ,\overline{\zeta }\in R^{m\times d},$ a.s.
where $L_i:\Delta ^c\rightarrow R$ is a deterministic function such
that the following holds:
\[
\sup\limits_{t\in [0,T]}\int_0^tL_1^p(s,t)ds<\infty ,\quad %
\sup\limits_{t\in [0,T]}\int_t^TL_2(t,s)^{2+\epsilon }ds<\infty
,\sup\limits_{t\in [0,T]}\int_0^tL_3^{\frac{2p}{2-p}}(s,t)ds<\infty
,
\]
where $\varepsilon \ $is a positive constant, and $p\in (1,2).$
\begin{theorem}
Let (H2) hold, then for any $\psi (\cdot )\in
L_{\mathcal{F}_T}^p[0,T]$, (1) admits a unique adapted M-solution in
${\mathcal{H}}^p[0,T],$ where $p\in (1,2).$
\end{theorem}
\begin{proof}We will split the proof into four steps.

Step 1. Choose $S\in [0,T]$ in a manner that we can determine the
unique existence of M-solution $(Y(t),Z(t,s))\in
{\mathcal{H}}^p[0,T]$ for $(t,s)\in [S,T]^2.$ First let
${\mathcal{M}}^p[0,T]$ be the space of all $(y(\cdot ),z(\cdot
,\cdot ))\in {\mathcal{H}}^p[0,T]$ such that
\begin{equation}
y(t)=Ey(t)+\int_0^tz(t,s)dW(s),\quad t\in [0,T].
\end{equation}
Thanks to the martingale moment inequalities in \cite{KS}, we deduce
that,
\begin{equation}
E\int_0^T\left| \int_0^tz(t,s)dW(s)\right| ^pdt\leq
C_pE\int_0^T\left( \int_0^t|z(t,s)|^2ds\right) ^{\frac p2}dt, \quad
 p>0,
\end{equation}
and
\begin{equation}
E\int_0^T\left( \int_0^t|z(t,s)|^2ds\right) ^{\frac p2}dt\leq
C_pE\int_0^T\left| \int_0^tz(t,s)dW(s)\right| ^pdt,\quad  p>1,
\end{equation}
where $C_p$ is a constant depending on $p.$ Thus it is easy to show
that ${\mathcal{M}}^p[0,T]$ is a closed nonempty subspace of
${\mathcal{H}}^p[0,T].$ (11) and (13) imply,
\[
E\int_0^T\left( \int_0^t|z(t,s)|^2ds\right) ^{\frac p2}dt\leq
C_pE\int_0^T|y(t)|^pdt,
\]
thus the following result holds,
\begin{eqnarray*}
&&\ E\int_0^T|y(t)|^pdt+E\int_0^T\left( \int_0^T|z(t,s)|^2ds\right)
^{\frac
p2}dt \\
\ &\leq &C_pE\int_0^T|y(t)|^pdt+C_pE\int_0^T\left(
\int_t^T|z(t,s)|^2ds\right) ^{\frac p2}dt\\
&&+C_pE\int_0^T\left(
\int_0^t|z(t,s)|^2ds\right) ^{\frac p2}dt \\
\ &\leq &C_pE\int_0^T|y(t)|^pdt+C_pE\int_0^T\left(
\int_t^T|z(t,s)|^2ds\right) ^{\frac p2}dt.
\end{eqnarray*}
Therefore, we can introduce another norm for ${\mathcal{M}}^p[0,T]$
as follows:
\[
\left\| ((y(\cdot ),z(\cdot ,\cdot ))\right\|
_{{\mathcal{M}}^p[0,T]}=\left[ E\int_0^T|y(t)|^pdt+E\int_0^T\left(
\int_t^T|z(t,s)|^2ds\right) ^{\frac p2}dt\right] ^{\frac 1p}.
\]
Let us consider the following equation:
\begin{equation}
Y(t)=\psi
(t)+\int_t^Tg(t,s,y(s),Z(t,s),z(s,t))ds-\int_t^TZ(t,s)dW(s),\quad
t\in [S,T]
\end{equation}
for any $\psi (\cdot )\in L_{{\mathcal{F}}_T}^p[S,T]$ and $(y(\cdot
),z(\cdot ,\cdot ))\in {\mathcal{M}}^p[S,T]$. By Proposition 2.2, we
observe that (14) admits a unique
adapted M-solution $(Y(\cdot ),Z(\cdot ,\cdot ))$, and we can define a map $%
\Theta :{\mathcal{M}}^p[S,T]\rightarrow {\mathcal{M}}^p[S,T]$ by
\[
\Theta (y(\cdot ),z(\cdot ,\cdot ))=(Y(\cdot ),Z(\cdot ,\cdot
)),\quad \forall (y(\cdot ),z(\cdot ,\cdot ))\in
{\mathcal{M}}^p[S,T].
\]
Let $(\overline{y}(\cdot ),\overline{z}(\cdot ,\cdot ))\in {\mathcal{M}}%
^p[S,T] $ and $\Theta (\overline{y}(\cdot ),\overline{z}(\cdot ,\cdot ))=(%
\overline{Y}(\cdot ),\overline{Z}(\cdot ,\cdot )).$ Consequently,
(9) gives that
\begin{eqnarray*}
&&E\int_S^T|Y(t)-\overline{Y}(t)|^pdt+E\int_S^T\left( \int_t^T|Z(t,s)-%
\overline{Z}(t,s)|^2ds\right) ^{\frac p2}dt \\
&\leq &CE\int_S^T\left\{ \int_t^T|g(t,s,y(s),Z(t,s),z(s,t))-g(t,s,\overline{y%
}(s),Z(t,s),\overline{z}(s,t))|ds\right\} ^pdt \\
&\leq &CE\int_S^T\left\{
\int_t^TL_1(t,s)|y(s)-\overline{y}(s)|ds\right\} ^pdt
\\
&&+CE\int_S^T\left\{
\int_t^TL_3(t,s)|z(s,t)-\overline{z}(s,t)|ds\right\} ^pdt
\\
&\leq &C\int_S^T(T-t)^{\frac pq}E\int_t^TL_1^p(t,s)|y(s)-y(s)|^pdsdt \\
&&+C\int_S^T(T-t)^{\frac pq}E\int_t^TL_3^p(t,s)|z(s,t)-z(s,t)|^pdsdt \\
&\leq &C(T-S)^{\frac pq}\sup_{t\in [0,T]}\int_0^tL_1^p(s,t)ds\cdot
E\int_S^T|y(s)-y(s)|^pds \\
&&+C(T-S)^{\frac pq}\sup_{t\in [0,T]}\left( \int_0^tL_3^{\frac{2p}{2-p}%
}(s,t)ds\right) ^{\frac{2-p}2}E\int_S^T\left( \int_0^t|z(t,s)-\overline{z}%
(t,s)|^2ds\right) ^{\frac p2}dt \\
&\leq &C(T-S)^{\frac pq}E\int_S^T|y(t)-\overline{y}(t)|^pdt.
\end{eqnarray*}
Then we can choose a constant $\eta =T-S$ so that $C\eta ^{\frac
pq}=\frac 12.$ Hence (14) admits a unique fixed point $(Y(\cdot
),Z(\cdot ,\cdot ))\in {\mathcal{M}}^p[S,T]$ which is the unique
adapted M-solution of equation (1) over $[S,T]$.

Step 2: We can apply the martingale representation theorem to
determine the value of $Z(t,s)$ for $(t,s)\in [S,T]\times [R,S]$
with $0<R<S,$ i.e.,
\[
E^{\mathcal{F}_S}Y(t)=E^{\mathcal{F}_R}Y(t)+\int_R^SZ(t,s)dW(s).
\]

Step 3: We determine the value of $Z(t,s)$ for $(s,t)\in [S,T]\times
[R,S]$ by solving a stochastic Fredholm integral equation, that is,
\begin{equation}
\psi ^S(t)=\psi (t)+\int_S^Tg^S(t,s,Z(t,s))ds-\int_S^TZ(t,s)dW(s),
\end{equation}
for $t\in [R,S],$ where
\[
g^S(t,s,z)=g(t,s,Y(s),z,Z(s,t)),
\]
with $(t,s,z)\in [R,S]\times [S,T]\times
L^p(S,T;L_{\Bbb{F}}^2[R,S]).$ From Proposition 2.3, we know that
(15) has a unique adapted solution $(\psi ^S(\cdot
),Z(\cdot ,\cdot ))\in L_{\mathcal{F}_S}^p[R,S]\times L^p(R,S;L_{\Bbb{F}%
}^2[S,T])$ with $\psi ^S(\cdot )$ being $\mathcal{F}_S$-measurable.

Step 4: We can complete the unique existence of adapted M-solution
by induction. \end{proof}

Let us consider the following BSVIE,
\begin{eqnarray}
Y(t)=\psi (t)+\int_t^Tg(t,s,Y(s),Z(t,s))ds-\int_t^TZ(t,s)dW(s),\quad
t\in [0,T].
\end{eqnarray}
When $\psi (\cdot )=\xi ,$ Lin \cite{L} studied the adapted solution
of (16) and Wang and Zhang \cite{WZ} considered BSVIE (16) with jump
under non-Lipschitz coefficient. We would like to mention the work
of results in Hilbert space in \cite{R1}. Recently, the authors
introduced a notion of S-solution of BSVIE (1) and they considered
the unique existence of adapted
solution by means of S-solution in \cite{WS}. Next we will give a general result in ${\mathcal{H}}%
_0^p[0,T],$ $p\in (1,2),$ which generalizes the above results. We
have,
\begin{theorem}
Let (H2) hold, we assume that,
\begin{eqnarray*}
\sup\limits_{t\in [0,T]}\int_t^TL_1^q(t,s)ds<\infty ,\quad   p>2, \\
\sup\limits_{t\in [0,T]}\int_0^tL_1^p(s,t)ds<\infty ,\quad  1<p<2,
\end{eqnarray*}
where $\frac 1p+\frac 1q=1.$ Then for any $\psi (\cdot )\in L_{\mathcal{F}%
_T}^p[0,T]$, (16) admits a unique adapted solution in
${\mathcal{H}}^p[0,T].$
\end{theorem}

\begin{proof}We split the proof into several steps.

Step 1: In this step, we will determine the value of $(Y(t),Z(t,s))$
for $(t,s)\in \Delta^c[S,T].$ We consider the following equation:
$t\in [S,T],$
\begin{equation}
Y(t)=\psi (t)+\int_t^Tg(t,s,y(s),Z(t,s))ds-\int_t^TZ(t,s)dW(s),
\end{equation}
for any $\psi (\cdot )\in L_{\mathcal{F}_T}^p[S,T]$ and $(y(\cdot
),z(\cdot ,\cdot ))\in {\mathcal{H}}_0^p[S,T]$. By Proposition 2.2,
we know that (17) admits a unique adapted solution $(Y(\cdot
),Z(\cdot ,\cdot ))\in {\mathcal{H}}_0^p[S,T]
$, and we can define a map $\Theta :{\mathcal{H}}_0^p[S,T]\rightarrow {\mathcal{%
H}}_0^p[S,T]$ by
\[
\Theta (y(\cdot ),z(\cdot ,\cdot ))=(Y(\cdot ),Z(\cdot ,\cdot
)),\quad \forall (y(\cdot ),z(\cdot ,\cdot ))\in
{\mathcal{H}}_0^p[S,T].
\]
Let $(\overline{y}(\cdot ),\overline{z}(\cdot ,\cdot ))\in {\mathcal{H}}%
_0^2[S,T]$ and $\Theta (\overline{y}(\cdot ),\overline{z}(\cdot ,\cdot ))=(%
\overline{Y}(\cdot ),\overline{Z}(\cdot ,\cdot )).$ By (9) we see
that,
\begin{eqnarray*}
&&\ E\int_S^T|Y(t)-\overline{Y}(t)|^pdt+E\int_S^T\left( \int_t^T|Z(t,s)-%
\overline{Z}(t,s)|^2ds\right) ^{\frac p2}dt \\
\  &\leq &CE\int_S^T\left\{ \int_t^T|g(t,s,y(s),Z(t,s))-g(t,s,\overline{y}%
(s),Z(t,s))|ds\right\} ^pdt \\
\  &\leq &CE\int_S^T\left\{
\int_t^TL_1(t,s)|y(s)-\overline{y}(s)|ds\right\} ^pdt.
\end{eqnarray*}
If $p\in (1,2),$ we arrive at,
\begin{eqnarray*}
&&CE\int_S^T\left\{ \int_t^TL_1(t,s)|y(s)-\overline{y}(s)|ds\right\} ^pdt \\
&\leq &C\int_S^T(T-t)^{\frac pq}E\int_t^TL_1^p(t,s)|y(s)-y(s)|^pdsdt \\
&\leq &(T-S)^{\frac pq}C\sup_{t\in [0,T]}\int_0^tL_1^p(s,t)ds\cdot
E\int_S^T|y(t)-y(t)|^pdt,
\end{eqnarray*}
and if $p>2,$
\begin{eqnarray*}
&&CE\int_S^T\left\{ \int_t^TL_1(t,s)|y(s)-\overline{y}(s)|ds\right\} ^pdt \\
&\leq &CE\int_S^T\left( \int_t^TL_1^q(t,s)ds\right) ^{\frac pq}\int_t^T|y(s)-%
\overline{y}(s)|^pds \\
&\leq &C(T-S)\sup_{t\in [0,T]}\left( \int_t^TL_1^q(t,s)ds\right)
^{\frac pq}E\int_S^T|y(t)-\overline{y}(t)|^pdt,
\end{eqnarray*}
where $\frac 1p+\frac 1q=1$. Then we can choose a constant $\eta
=T-S$ so that $C\max \{\eta ^{\frac pq},\eta \}=\frac 12.$ Hence
(17) admits a unique fixed point $(Y(\cdot ),Z(\cdot ,\cdot ))\in
{\mathcal{H}}_0^p[S,T]$ which is the unique adapted solution of
equation (16) over $[S,T]$.

Step 2: We determine the value of $Z(t,s)$ for $(s,t)\in [S,T]\times
[R,S]$ by solving a stochastic Fredholm integral equation, that is,
\begin{equation}
\psi ^S(t)=\psi (t)+\int_S^Tg^{S}(t,s,Z(t,s))ds-\int_S^TZ(t,s)dW(s),
\end{equation}
for $t\in [R,S],$ where $ g^{S}(t,s,z)=g(t,s,Y(s),z).$ From
Proposition 2.3, we know that (18) admits a unique adapted solution
$(\psi ^S(\cdot ),Z(\cdot ,\cdot ))\in L_{\mathcal{F}_S}^p[R,S]\times L^p(R,S;L_{\Bbb{F}}^2[S,T])$ with $\psi ^S(\cdot )$ being $\mathcal{F%
}_S$-measurable.

Step 3: We can complete the unique existence of adapted solution by
induction.
\end{proof}

\end{document}